\documentclass[reqno,12pt]{amsart}
\usepackage{mathrsfs}
\usepackage{amscd}
\usepackage{amsmath}
\usepackage{latexsym}
\usepackage{amsfonts}
\usepackage{amssymb}
\usepackage{amsthm}
\usepackage{graphicx}

\numberwithin{equation}{section}

\def\H{{\cal H}}

\def\R{\mathbb{R}}

\def\H1{H^1(\R)}

\newtheorem{thm}{Theorem}[section]
\newtheorem{lem}{Lemma}[section]

\makeatletter
\newcommand{\Extend}[5]{\ext@arrow0099{\arrowfill@#1#2#3}{#4}{#5}}
\makeatother

\begin{document}

\setcounter{page}{1}

\title[Global Well-posedness for DNLS]{Global well-posedness on the derivative nonlinear
Schr\"{o}dinger equation}

\author{Yifei Wu}
\address{School of  Mathematical Sciences, Beijing Normal University, Laboratory of Mathematics and Complex Systems,
Ministry of Education, Beijing 100875, P.R.China}
\email{yifei@bnu.edu.cn}
\thanks{The author was partially supported
by the NSF of China (No. 11101042), and the Fundamental Research Funds for the Central Universities of China.}

\subjclass[2010]{Primary  35Q55; Secondary 35A01}


\keywords{Nonlinear Schr\"{o}dinger equation with derivative,
global well-posedness, energy space}

\maketitle

\begin{abstract}\noindent
As a continuation of the previous work \cite{Wu}, we consider the global well-posedness for the derivative nonlinear
Schr\"{o}dinger equation. We prove that it is globally well-posed in energy space, provided that the initial data $u_0\in H^1(\R)$ with $\|u_0\|_{L^2}< 2\sqrt{\pi}$.
\end{abstract}

\section{Introduction}
We study the following Cauchy problem of the nonlinear
Schr\"{o}dinger equation with derivative (DNLS):
 \begin{equation}\label{eqs:DNLS}
   \left\{ \aligned
    &i\partial_{t}u+\partial_{x}^{2}u=i\partial_x(|u|^2u),\qquad t\in \R, x\in \R,
    \\
    &u(0,x)  =u_0(x)\in   H^1(\mathbb{R}).
   \endaligned
  \right.
 \end{equation}
It arises from studying the propagation of
circularly polarized Alfv\'{e}n waves in magnetized plasma with a
constant magnetic field, see \cite{MOMT-PHY, M-PHY, SuSu-book} and the references therein.
The problem (\ref{eqs:DNLS}) is $L^2$-critical, and the equation is complete integrable.  Moreover,
the $H^1$-solution of (\ref{eqs:DNLS}) obeys the following mass,  energy, and momentum conservation laws,
\begin{align}
M_D(u(t))&:=\int_{\R}|u(t)|^2\,dx=M_D(u_0),\label{mass-law}\\
E_D(u(t))&:= \int_{\R}\Big(| u_x(t)|^2+\frac{3}{2}\text{Im}
|u(t)|^2u(t)\overline{u_x(t)}+\frac{1}{2}|u(t)|^6\Big)\,dx=E_D(u_0),\label{energy-law}\\
P_D(u(t))&:=\mbox{Im} \int_{\R}\bar u(t) u_x(t)\,dx-\frac 12\int_{\R}
|u(t)|^4\,dx=P_D(u_0).\label{mum-law}
\end{align}

Local well-posedness for the Cauchy problem (\ref{eqs:DNLS}) is
well-understood. It was proved in the energy space $H^1(\R)$ by Hayashi and Ozawa in \cite{Ha-93-DNLS, HaOz-92-DNLS,
HaOz-94-DNLS}, see also Guo and Tan \cite{GuTa91} for earlier result
in smooth spaces. The problem for rough data below the energy space, see \cite{BiLi-01-Illposed-DNLS-BO, Ta-99-DNLS-LWP,Ta-01-DNLS-GWP} for local well-posedness and ill-posedness.

The global well-posedness for (\ref{eqs:DNLS}) has been also widely
studied.
By using mass and energy conservation laws, and by developing the gauge
transformations, Hayashi and Ozawa \cite{HaOz-94-DNLS, Oz-96-DNLS}
proved that the problem (\ref{eqs:DNLS}) is globally well-posed in energy space
$H^1(\R)$ under the condition
\begin{equation}\label{small condition}
\|u_0\|_{L^2}<\sqrt{2\pi }.
\end{equation}
Here $2\pi$ is the mass of the ground state $Q$ of the corresponding elliptic problem. That is,  $Q$ is the unique (up to some
symmetries) positive solution of the following elliptic equation
\begin{equation}\label{elliptic}
-Q_{xx}+Q-\frac 3{16}Q^5=0.
\end{equation}
Since $\sqrt{2\pi }$ is just the mass of the ground state,
the condition of \eqref{small condition} was naturally used before to keep the energy to be positive.

However, the condition \eqref{small condition} was improved recently in our previous work \cite{Wu}. We prove that there exists a small constant $\varepsilon_*>0$, such that the problem (\ref{eqs:DNLS}) is still globally well-posed in energy space when the initial data satisfies $\|u_0\|_{L^2}<\sqrt{2\pi }+\varepsilon_*.$ The result implies that for the problem \eqref{eqs:DNLS}, the ground state mass $2\pi$ is not the threshold of the global well-posedness and blow-up. This is totally different from the $L^2$-critical power-type  Schr\"odinger equation (the nonlinearity $i\partial_x(|u|^2u)$ in (\ref{eqs:DNLS}) replaced by $-\frac{3}{16}|u|^4u$), see \cite{Wu} for some further discussion.
Results on the global well-posedness when the initial data of regularity below the energy space, see \cite{CKSTT-01-DNLS, CKSTT-02-DNLS, Miao-Wu-Xu:2011:DNLS,Ta-01-DNLS-GWP} for examples.

In this paper,  we continue to consider the $L^2$-assumption on initial data and obtain the global well-posedness as follows.
\begin{thm}\label{thm:main}  For any $u_0\in \H1$
with
\begin{equation}\label{small condition-further}
\int_{\R}|u_0(x)|^2\,dx <4\pi,
\end{equation}
the Cauchy problem (\ref{eqs:DNLS}) is globally
well-posed in $H^1(\R)$ and the solution $u$ satisfies
$$
\|u\|_{L^\infty_t H^1_x}\le C(\|u_0\|_{H^1}).
$$
\end{thm}
Since $4\pi= 2\|Q\|_{L^2}^2$, we in fact prove that the Cauchy problem (\ref{eqs:DNLS}) is globally
well-posed in $H^1(\R)$,  when $\int_{\R}|u_0(x)|^2\,dx< 2\int_{\R}Q(x)^2\,dx$.


Developing by Hayashi and Ozawa, the gauge transformation is one of element tools to study the derivative nonlinear  Schr\"odinger equation. Let
\begin{equation}\label{gauge1}
v(t,x):=e^{-\frac{3}{4}i\int_{-\infty}^x |u(t,y)|^2\,dy}u(t,x),
\end{equation}
then from \eqref{eqs:DNLS}, $v$ is the solution of
\begin{equation}\label{eqs:DNLS-under-gauge1}
     i\partial_{t}v+\partial_{x}^{2}v=\frac i2|v|^2v_x-\frac i2v^2\bar{v}_x-\frac
     3{16}|v|^4v,
\end{equation}
with the initial data $v_0=e^{-\frac{3}{4}i\int_{-\infty}^x |u_0|^2\,dy}u_0$.
Moreover, $v$ obeys the same mass conservation law \eqref{mass-law}, while
the energy conservation law \eqref{energy-law}  is deduced to
\begin{equation}\label{energy-43}
E(v(t)):=\|v_x(t)\|_{L^2_x}^2-\frac{1}{16}\|v(t)\|^6_{L^6_x}=
E(v_0),
\end{equation}
and the  momentum conservation law \eqref{mum-law} is deduced to
\begin{equation}\label{mom-law}
P(v(t)):=\mbox{Im} \int_{\R}\bar v(t) v_x(t)\,dx+\frac 14\int_{\R}
|v(t)|^4\,dx=P(v_0).
\end{equation}

From the argument used in \cite{Wu}, to prove the global well-posedness for the DNLS, an important ingredient is the usage of the momentum conservation law.
We observe that the key point is to give a small control of the following term from \eqref{mom-law},
\begin{align}
\mbox{Im} \int_{\R}\bar v(t) v_x(t)\,dx.
\end{align}
Or, to be more exact, one may prove that
\begin{align}\label{key-bound}
-\mbox{Im} \int_{\R}\bar v(t) v_x(t)\,dx\le c\|v_x(t)\|_{L^2 }\|v(t)\|_{L^2 },
\end{align}
where $c$ is positive constant. It is trivial when $c=1$ by H\"older's inequality. Suppose that one can get the bound with a
suitable small constant $c$, then the global well-posedness were followed.
In \cite{Wu}, by using a variational argument,
we in fact proved  that if the mass is larger but close to $2\pi$, and there is a time sequence $\{t_n\}$ such that $\|v(t_n)\|_{H^1}$ tends to infinity,
then $v(t_n)$ is close to $Q$ up to some symmetries. Since $Q$ is real-valued, \eqref{key-bound} can be given for small $c>0$.

In this paper, we give a different argument to prove the bound \eqref{key-bound},
under some suitable but explicit assumption  on $L^2$-norm of the initial data. Our method here do not need to use the property of the ground state $Q$ of \eqref{elliptic}.
To give an explanation on our argument, we say that
if $\|v(t)\|_{H^1}$ tends to infinity, then by the momentum and energy conservation laws, \eqref{key-bound} is roughly deduced to
\begin{align*}
\frac14\|v(t)\|_{L^4}^4\thickapprox-\mbox{Im} \int_{\R}\bar v(t) v_x(t)\,dx\le c\|v_x(t)\|_{L^2 }\|v(t)\|_{L^2 }\thickapprox c\|v_0\|_{L^2 }\|v\|_{L^6 }^3.
\end{align*}
So to obtain the small bound $c$, we turn to obtain the smallness of the quantity
$$f(t):=\|v(t)\|_{L^4}^4\big/\|v(t)\|_{L^6 }^3.
$$
Indeed, we shall prove that $f(t)^2$ obeys some cubic inequality.
Then we find that, the settlement to the global well-posedness is turned to find the solution to an elementary cubic equation.

%
%

%

\vspace{0.5cm}
\section{The proof of Theorem \ref{thm:main}}

Let $v$ be  the function in \eqref{gauge1}, which is the solution of the equation \eqref{eqs:DNLS-under-gauge1}. Since
$$
u_x=e^{i\frac34 \int_{-\infty}^x |v(t,y)|^2\,dy}\big(i\frac34|v|^2v+v_x\big).
$$
Therefore, by the sharp Gagliardo-Nirenberg inequality (see
\cite{W}),
\begin{equation}\label{sharp Gagliardo-Nirenberg}
\|f\|_{L^6}^6\leq \frac{4}{\pi^2}\|f\|_{L^2}^4\|f_x\|_{L^2}^2,
\end{equation}
(which the equality is attained by $Q$), and mass conservation law, for any $t\in \R$,
\begin{align*}
\|u_x(t)\|_{L^2}\le & \|v_x(t)\|_{L^2}+\frac34\|v(t)\|_{L^6}^3
\le  \|v_x(t)\|_{L^2}+\frac3{2\pi}\|v(t)\|_{L^2}^2\|v_x(t)\|_{L^2}\\
\le & \big(1+\frac3{2\pi}\|u_0\|_{L^2}^2\big)\|v_x(t)\|_{L^2}.
\end{align*}
That is,  the boundedness of $u$ in $H^1$-norm is equivalent to the boundedness of $v$ in $H^1$-norm.
Therefore, to prove the theorem, we may consider the function $v$ in \eqref{gauge1} instead. To simply the notations,
from now on, we set
$$
E_0=E(v_0),\quad P_0=P(v_0), \quad m_0=M_D(v_0).
$$
Furthermore, we assume that $m_0>2\pi$. Otherwise, it has been proved the global well-posedness in \cite{HaOz-94-DNLS, Wu}.

Let $(-T_-(v_0),T_+(v_0))$ be the maximal lifespan of the solution $v$ of \eqref{eqs:DNLS-under-gauge1}. To prove  Theorem \ref{thm:main}, it is  sufficient to
obtain the (indeed uniformly)\emph{ a priori} estimate of the
solutions on $H^1$-norm, that is,
$$
\sup\limits_{t\in (-T_-(v_0),T_+(v_0))}\|v_x(t)\|_{L^2}< +\infty.
$$
As \cite{Wu}, we argue by contradiction. Suppose that there exists a
sequence $\{t_n\}$ with
$
t_n\to -T_-(v_0)$,  or   $T_+(v_0)$,
such that
\begin{equation}\label{infty-sequence}
\|v_x(t_n)\|_{L^2}\to +\infty,   \mbox{ as  } n\to\infty.
\end{equation}

Now we define the sequence $\{f_n\}$ by
$$
f_n=\frac{\|v(t_n)\|_{L^4}^4}{\|v(t_n)\|_{L^6}^3},
$$
then $\|v(t_n)\|_{L^4}^4=f_n \|v(t_n)\|_{L^6}^3$, and by the energy conservation law \eqref{energy-43},
\begin{equation}\label{15.48}
\|v(t_n)\|_{L^4}^8=16 f_n^2 \big(\|v_x(t_n)\|_{L^2}^2-E_0\big).
\end{equation}
Moreover,  we give the lower and upper bounds of $f_n$. We denote $C_{GN}$ to be  the sharp constant of the following Gagliardo-Nirenberg inequality,
\begin{equation}\label{sharp Gagliardo-Nirenberg2}
\|f\|_{L^6}\leq C_{GN} \|f\|_{L^4}^{\frac89}\|f_x\|_{L^2}^{\frac19}.
\end{equation}
The best constant and the optimal functions for this inequality was obtained by Agueh \cite{Agueh}. More precisely, an optimal function is written as
$$
\Psi(x)=\Big(x^2+1\Big)^{-\frac12},
$$
and the best constant $C_{GN}=3^{\frac16}(2\pi)^{-\frac19}$. In particular, $\sqrt 2\Psi$ is the unique (up to symmetries) radial ground state of the following elliptic equation,
$$
\partial_{xx}\psi-\psi^3+\frac34 \psi^5=0.
$$
Now we have
\begin{lem}\label{lem:fn-bound}
There exists an $\varepsilon_n: \varepsilon_n\to 0$ as $n\to \infty$, such that
\begin{align}
2C_{GN}^{-\frac92}+\varepsilon_n\le f_n\le \sqrt{m_0}.
\end{align}
\end{lem}

\begin{proof}[Proof of Lemma \ref{lem:fn-bound}]
First, by H\"older's inequality, we have
$$\|v(t_n)\|_{L^4}^4\le \|v(t_n)\|_{L^2}\|v(t_n)\|_{L^6}^3=\sqrt{m_0}\|v(t_n)\|_{L^6}^3,$$
and thus
$$
f_n\le \sqrt{m_0}.
$$
Furthermore, by the sharp Gagliardo-Nirenberg  \eqref{sharp Gagliardo-Nirenberg2} and the energy conservation law \eqref{energy-43}, we have
\begin{align*}
f_n&\ge \frac{\Big( C_{GN}^{-6}\|v(t_n)\|_{L^6}^6\>\|v_x(t_n)\|_{L^2}^{-\frac23}\Big)^{\frac34}}{\|v(t_n)\|_{L^6}^3}\notag\\
=&C_{GN}^{-\frac92}\frac{\|v(t_n)\|_{L^6}^{\frac32}}{\|v_x(t_n)\|_{L^2}^{\frac12}}
=2C_{GN}^{-\frac92}\frac{\|v(t_n)\|_{L^6}^{\frac32}}{\Big(\|v(t_n)\|_{L^6}^6+16E_0\Big)^{\frac14}}\notag\\
=&2C_{GN}^{-\frac92}+\varepsilon_n,
\end{align*}
where
$$
\varepsilon_n:=2C_{GN}^{-\frac92}\|v(t_n)\|_{L^6}^{\frac32}\frac{\|v(t_n)\|_{L^6}^{\frac32}-\big(\|v(t_n)\|_{L^6}^6+16E_0\big)^{\frac14}}{\|v(t_n)\|_{L^6}^{\frac32}\big(\|v(t_n)\|_{L^6}^6+16E_0\big)^{\frac14}}
.$$
Note that $\varepsilon_n=O(\|v_n\|_{L^6}^{-6})\to 0$, by the mean value theorem.
\end{proof}
In spirit of the paper \cite{Ba-ASNSP-04}, we define
$$
\phi(t,x)= e^{i\alpha x}v(t,x),
$$
where the parameter $\alpha$ will be given later. Then $\phi_x(t,x)= e^{i\alpha x}\big(i\alpha v(t,x)+v_x(t,x)\big)$, and thus
\begin{align*}
\|\phi_x\|_{L^2 }^2= \|v_x\|_{L^2 }^2+2\alpha\mbox{Im}\int \bar v\>v_x\,dx+\alpha^2\|v\|_{L^2 }^2.
\end{align*}
Minus $\frac1{16}\|\phi\|_{L^6 }^6=\frac1{16}\|v\|_{L^6 }^6$ in the two sides,  yields that
\begin{align*}
E(\phi)= E(v)+2\alpha\mbox{Im}\int \bar v\>v_x\,dx+\alpha^2\|v\|_{L^2 }^2.
\end{align*}
By mass, energy conservation laws \eqref{mass-law} and \eqref{energy-43}, this gives that
\begin{align}\label{Ene-Ene}
-2\alpha\mbox{Im}\int \bar v\>v_x\,dx=-E(\phi)+\alpha^2m_0+E_0.
\end{align}
On the other hand, by using \eqref{sharp Gagliardo-Nirenberg2}, we have
\begin{align*}
E(\phi(t_n))=&\|\phi_x(t_n)\|_{L^2 }^2-\frac1{16}\|\phi(t_n)\|_{L^6 }^6\\
\ge &C_{GN}^{-18}\|\phi(t_n)\|_{L^6 }^{18}\|\phi(t_n)\|_{L^4 }^{-16}-\frac1{16}\|\phi(t_n)\|_{L^6 }^6\\
=&\Big(C_{GN}^{-18}\|v(t_n)\|_{L^6 }^{12}\|v(t_n)\|_{L^4 }^{-16}-\frac1{16}\Big)\|\phi(t_n)\|_{L^6 }^6\\
=&\Big(C_{GN}^{-18}f_n^{-4}-\frac1{16}\Big)\|v(t_n)\|_{L^6 }^6.
\end{align*}
Therefore, this combining with \eqref{Ene-Ene},  gives that
\begin{align*}
-2\alpha\mbox{Im}\int \bar v(t_n,x)\>v_x(t_n,x)\,dx\le \Big(\frac1{16}-C_{GN}^{-18}f_n^{-4}\Big)\|v(t_n)\|_{L^6 }^6+\alpha^2m_0+E_0.
\end{align*}
This implies that for $\alpha>0$,
\begin{align}\label{15.56}
-\mbox{Im}\int \bar v(t_n,x)\>v_x(t_n,x)\,dx\le \Big(\frac1{16}-C_{GN}^{-18}f_n^{-4}\Big)\|v(t_n)\|_{L^6 }^6\cdot\frac1{2\alpha}+\frac12\alpha m_0+ \frac12\alpha^{-1}E_0.
\end{align}
We consider the case of $\frac1{16}-C_{GN}^{-18}f_n^{-4}< 0$ first. Then by the momentum conservation law \eqref{mom-law}, we have
\begin{align}\label{15.50}
\frac14 \|v(t_n)\|_{L^4}^4=-\mbox{Im}\int \bar v(t_n,x)\>v_x(t_n,x)\,dx+P_0.
\end{align}
Hence by \eqref{15.56} and choosing $\alpha=1$, we obtain
\begin{align*}
\|v(t_n)\|_{L^4}^4\le 2(m_0+E_0+2P_0).
\end{align*}
Therefore, by \eqref{15.48} and Lemma \ref{lem:fn-bound}, we have the boundedness of $\|v_x(t_n)\|_{L^2}$. This is a contradiction with \eqref{infty-sequence}.

Now we consider the case of $\frac1{16}-C_{GN}^{-18}f_n^{-4}\ge 0$.  We set
$$
\alpha=\frac14\sqrt{m_0^{-1}\Big(1-16C_{GN}^{-18}f_n^{-4}\Big)}\|v(t_n)\|_{L^6}^3,
$$
then $\alpha=\alpha_n\to \infty$ as $n\to \infty$, and \eqref{15.56} turns to
\begin{align}
-\mbox{Im}\int \bar v(t_n,x)\>v_x(t_n,x)\,dx\le \frac14\sqrt{m_0\Big(1-16C_{GN}^{-18}f_n^{-4}\Big)}\|v(t_n)\|_{L^6 }^3+\frac12\alpha_n^{-1}E_0. \label{Key-1}
\end{align}
Note that the remainder term
$$
\frac12\alpha_n^{-1}E_0\to 0,\quad \mbox{ as }\quad  n\to \infty.
$$
By \eqref{15.50} and \eqref{Key-1},
\begin{align*}
\|v(t_n)\|_{L^4}^4\le \sqrt{m_0\Big(1-16C_{GN}^{-18}f_n^{-4}\Big)}\|v(t_n)\|_{L^6 }^3+2\alpha_n^{-1}E_0+4P_0.
\end{align*}
This implies that
\begin{align*}
f_n\le \sqrt{m_0\Big(1-16C_{GN}^{-18}f_n^{-4}\Big)}+\big(2\alpha_n^{-1}E_0+4P_0\big)\|v(t_n)\|_{L^6 }^{-3}.
\end{align*}
Furthermore, it gives that
\begin{align}\label{inequality1}
f_n^6\le m_0f_n^4-16m_0C_{GN}^{-18}+f_n^4\mathcal R_n,
\end{align}
where
$$
\mathcal R_n=2 \sqrt{m_0\Big(1-16C_{GN}^{-18}f_n^{-4}\Big)}\>\big(2\alpha_n^{-1}E_0+4P_0\big)\|v(t_n)\|_{L^6 }^{-3}
+\big(2\alpha_n^{-1}E_0+4P_0\big)^2\|v(t_n)\|_{L^6 }^{-6}.
$$
By the lower and upper boundedness of $f_n$ from Lemma \ref{lem:fn-bound}, we have
$$
f_n^4\mathcal R_n=O(\|v(t_n)\|_{L^6 }^{-3})\to 0,\quad \mbox{ as } n\to \infty.
$$
Thus for any small and  fixed $\epsilon>0$, by choosing $n$ large enough, we have $f_n^4\mathcal R_n<\epsilon$.
Hence \eqref{inequality1} becomes
\begin{align}\label{inequality}
f_n^6\le m_0f_n^4-16m_0C_{GN}^{-18}+\epsilon.
\end{align}
Let $X=f_n^2$, then \eqref{inequality} turns to the inequality
\begin{align}\label{11.08}
X^3-m_0X^2+b<0,
\end{align}
where $b=16m_0C_{GN}^{-18}-\epsilon>0$.  Let
$$
F(X)=X^3-m_0X^2+b,
$$
then the function $F(X)$ attains its minimum value at $\frac23 m_0$ in the region of $[0,\infty)$. Therefore,
there are two positive solutions $X_1$ and $X_2$ solve  the equation
\begin{align}\label{cubicequation}
X^3-m_0X^2+b=0.
\end{align}
if and only if
$
F(\frac23m_0)<0.
$
In another word, the inequality \eqref{11.08} has no solution in the region of $[0,+\infty)$ if and only if
\begin{align}
F(\frac23m_0)\ge 0.\label{11.27}
\end{align}

Now the condition \eqref{11.27} is equivalent to
$$
\frac{8}{27}m_0^3-\frac49m_0^3+b\ge 0.
$$
By the arbitrariness of $\epsilon$,  it is deduced to $m_0< 6\sqrt{3}C_{GN}^{-9}=4\pi$. Therefore, we obtain that the problem \eqref{eqs:DNLS-under-gauge1} is global well-posedness when $m_0<4\pi$. This proves our main theorem.

One may expect to get some profit from the restriction $X \in (4C_{GN}^{-9},m_0)$ (rather than $[0,+\infty)$), according to Lemma \ref{lem:fn-bound}.
However, we will explain below that we in fact can not get any more from this. To see this, we note that
in this case of  $m_0\ge 4\pi$, \eqref{11.08} is solved in the region of $[0,+\infty)$ by
$$
X_1<X<X_2.
$$
Now we claim that
\begin{align}
4C_{GN}^{-9}<X_1<X_2<m_0.\label{15.09}
\end{align}

Indeed, first we observe that when $m_0\ge 4\pi$,
$$
\frac23m_0\ge \frac83\pi> 4C_{GN}^{-9}=\frac8{3\sqrt3}\pi.
$$
Moreover, by choosing $\epsilon$ small, we have
$$
F(4C_{GN}^{-9})=64C_{GN}^{-27}-\epsilon>0.
$$
These two facts imply that $4C_{GN}^{-9}<X_1$. Similarly, since
$$
m_0<\frac23m_0;\quad \mbox{ and }\quad F(m_0)=b>0,
$$
we have
$X_2<m_0$. In conclusion, we have
the claim \eqref{15.09}.
Therefore, the inequality \eqref{11.08} is always solvable in the region of $(4C_{GN}^{-9},m_0)$ when $m_0\ge 4\pi$. So we can not get the contradiction from the restriction of $(4C_{GN}^{-9},m_0)$.

\section*{Acknowledgements} The author would like to thank Professor Martial Agueh for some useful discussion on the sharp Gagliardo-Nirenberg inequality.

\end{document}